\documentclass[11pt]{amsart}   
\pdfoutput=1

\usepackage{verbatim}
\usepackage{amssymb,amsthm,amsmath}
\usepackage{color}
\definecolor{darkblue}{rgb}{0, 0, .4}
\definecolor{grey}{rgb}{.7, .7, .7}

\usepackage{ifpdf}
\ifpdf
  \usepackage[pdftex]{epsfig}

  \usepackage{lscape}

  \usepackage[breaklinks]{hyperref}
  \hypersetup{
            colorlinks=true,
            linkcolor=darkblue,
            anchorcolor=darkblue,
            citecolor=darkblue,
            pagecolor=darkblue,
            urlcolor=darkblue,
            pdftitle={},
            pdfauthor={}
  }
\else
  \usepackage[dvips]{epsfig}
  \usepackage{lscape}

  \newcommand{\href}[2]{#2}
  \newcommand{\url}[2]{#2}
\fi

\newtheorem{theorem}{Theorem}[section]
\newtheorem{lemma}[theorem]{Lemma}

\theoremstyle{definition}
\newtheorem{definition}[theorem]{Definition}
\newtheorem{example}[theorem]{Example}

\theoremstyle{remark}
\newtheorem{remark}[theorem]{Remark}

\numberwithin{equation}{section}

\theoremstyle{theorem}
\newtheorem{corollary}[theorem]{Corollary}

\newcommand{\s}[0]{\sigma}

\usepackage{graphicx}
\input xy 
\xyoption{all}

\usepackage{tikz}
\usetikzlibrary{trees}
\usepackage{ifthen}

\newcommand{\ra}{\rightarrow}

\newcommand{\I}{\mathcal{I}}
\newcommand{\J}{\mathcal{J}}
\newcommand{\SN}{\mathfrak{S}}

\newcommand{\pf}[1]{\pi|_{[#1]}}

\renewcommand{\P}[0]{\mathcal{P}}

\newcommand{\q}{\rule{0.20cm}{0.15mm}}

\newcommand{\bchoose}[2]{\left[ { {#1} \atop {#2} } \right]}

\begin{document}

\title{Strategy-indifferent games of best choice}

\begin{abstract}
The game of best choice (or ``secretary problem'') is a model for making an
irrevocable decision among a fixed number of candidate choices that are
presented sequentially in random order, one at a time.  Because the classically
optimal solution is known to reject an initial sequence of candidates, a
paradox emerges from the fact that candidates have an incentive to position
themselves immediately after this cutoff which challenges the assumption that
candidates arrive in uniformly random order.

One way to resolve this is to consider games for which every (reasonable)
strategy results in the same probability of success.  In this work, we classify
these ``strategy-indifferent'' games of best choice.  It turns out that the
probability of winning such a game is essentially the reciprocal of the
expected number of left-to-right maxima in the full collection of candidate
rank orderings.  We present some examples of these games based on avoiding
permutation patterns of size $3$, which involves computing the distribution of
left-to-right maxima in each of these pattern classes.
\end{abstract}

\author[B. Jones]{Brant Jones}
\address{Department of Mathematics and Statistics, MSC 1911, James Madison University, Harrisonburg, VA 22807}
\email{\href{mailto:jones3bc@jmu.edu}{\texttt{jones3bc@jmu.edu}}}
\urladdr{\url{http://educ.jmu.edu/\~jones3bc/}}

\author[K. Kochalski]{Katelynn D. Kochalski}
\address{SUNY Geneseo}
\email{kochalski@geneseo.edu}

\author[S. Loeb]{Sarah Loeb}
\address{Hampden--Sydney College}
\email{sloeb@hsc.edu}

\author[J. Walk]{Julia C. Walk}
\address{Concordia College}
\email{jwalk@cord.edu}


\date{\today}

\maketitle

\section{Introduction}

The ``secretary problem'' or ``game of best choice'' is a model for sequential
decision making that was widely popularized in a 1960 column of Martin Gardner
(reprinted in \cite{gardner}) and has been generalized in many directions by
researchers working in probability and stochastic processes (see
\cite{gilbert--mosteller,ferguson,freeman} and the references therein).  In the
original setup, a player conducts interviews with a fixed number of candidates
with the goal of selecting the best.  After each interview, the player ranks
the current candidate against all of the candidates that have been considered
so far, obtaining a permutation.  The player must then decide whether to accept
the current candidate which ends the game or, alternatively, whether to reject
the current candidate forever and continue playing in hope of obtaining an even
better candidate in the future.  If we assume that the rank orders of the
candidates are distributed uniformly then it turns out that the optimal
strategy is to always reject an initial set of candidates, of size $N/e$ when
$N$ is large, and use them as a training set by selecting the next candidate
who is better than all of them (or the last candidate if no subsequent
candidate is better).  It can be shown that the probability of hiring the best
candidate out of $N$ with this strategy approaches $1/e$. 

More recently, researchers have begun applying the best-choice framework to
online auctions \cite{kleinberg08,kleinberg--etal} where the ``candidate
rankings'' are bids (that may arrive and expire at different times) and the
player must choose which bid to accept (ending the auction).  In this case,
there is an ``incentive paradox'' because candidates who know the optimal
strategy have a strong incentive to maneuver their bid away from the first
$1/e$ slots to avoid certain rejection, yet the permutation of candidate
rankings for each position were already assumed to be uniformly distributed.

Although there is a vast literature on the secretary problem, as far as we know
there are only two papers that address this issue.  In \cite{buchbinder--etal},
the authors use linear programming to give a {\em randomized mechanism on the
complete uniform distribution for which the probability of selecting a
candidate is independent of their position}.  While it avoids the incentive
paradox, this mechanism also has a lower probability of selecting the best
candidate than the classical solution.

Recently, in \cite{jones18} (see also \cite{fowlkes}), the first author of this
paper found, essentially by accident, a natural family of combinatorially
defined uniform distributions for which {\em every deterministic strategy
produces the same probability of success}.  This includes some subtle
non-positional strategies that are based on the relative ranking of the
candidates encountered during the interview process, and so we call such games
{\bf strategy-indifferent}.  These game variations avoid the incentive paradox
because any strategy could be employed by the interviewer (perhaps selected at
random), and it will necessarily be optimal.  So, as long as the interviewer's
strategy remains hidden, the candidates have no incentive to change their
interview position.  The asymptotic probability of success for the mechanisms
in each of these two (unrelated) papers is $1/4$, but in light of our work here
this seems to be merely coincidental.

Since the discovery of the $231$-avoiding game trees from \cite{jones18} with
their highly symmetric structure was surprising, and revealed only after a
combinatorial (as opposed to probabilistic) analysis of the game, we wondered
whether there were other examples of this phenomenon and whether it were
possible to define a game whose win probability was as high as the classical
game of best choice.

In this paper, we systematically examine the structure behind
strategy-indifferent games.  More precisely, in Section~\ref{s:games} we
explain how the strategy-indifferent games of best choice in rank $N$ are
classified by subsets of the symmetric group $\SN_{N-1}$ with win probabilities
determined by the mean number of left-to-right maxima per permutation in the
subset.  In Section~\ref{s:lrms} we provide some examples, including the
$231$-avoiding game from \cite{jones18}, a description of the largest
multiplicity-free strategy-indifferent game, and the win probabilities for the
games avoiding each of the other permutation patterns of size $3$.

Essentially, we find (see Corollary~\ref{c:const}) that the
strategy-indifferent games are those for which {\em the best candidate arrives
at one of the positions of a best-so-far candidate in the interview rank
ordering that would result if the interviewer rejected all but the last
candidate}.  Although they differ from the classical game of best choice (in
which the best candidate is equally likely to appear at any position), such
games do have a kind of uniformity in that the best candidate is equally likely
to appear at any {\em left-to-right maximum} position in the permutation
determined by the rankings of the competitors to the best candidate.  Whether
it is empirically true that best candidates should tend to occur concurrently
with a counterfactual best-so-far candidate is not known, but this could in
principle be tested and it does not seem unreasonable to us to imagine that, in
some contexts, the best candidate would tend to occur together with other
highly ranked candidates.  To the extent this is true, our results allow us to
recognize such games as strategy-indifferent.  One potential direction for
further analysis would be to develop an approximation theory for
strategy-indifferent games, in which one could measure precisely how failures
of the symmetry from our characterization affect the variance of win
probabilities among the strategies that a player could reasonably adopt.

As part of our analysis in Section~\ref{s:lrms}, we also compute the
distributions of left-to-right maxima among each of the sets of permutations
avoiding a single permutation pattern of size $3$ which, while not difficult,
may be of independent interest.  It turns out (see Figure~\ref{f:results}) that there
are three distinct equivalence classes of patterns, with nice relationships to
refinements of the Catalan numbers.  Thus, our work ties into a recent trend in
which various authors have investigated the discrete probability distribution
of some permutation statistic for a random model in which a pattern-avoiding
permutation is chosen uniformly at random. See, for example, \cite{miner-pak,
mp-312, dhw, fmw}.  Earlier, Wilf has collected some results on distributions
of left-to-right maxima for the full symmetric group in \cite{wilf} and
Prodinger \cite{prodinger} has studied these under a geometric random model.
Finally, this work also follows in a sequence of recent papers
\cite{weighted,milenkovic2,milenkovic1} studying generalized games of best
choice from the combinatorial perspective.

\section{Strategy-indifferent games}\label{s:games}

\subsection{Games of best choice}
Refer to \cite{bona} for an excellent introduction to permutation pattern
avoidance as well as the partially ordered set terminology we will be using.
We recommend \cite{jones18} for a more expository introduction to the game of
best choice.  Formally, each subset $\I$ of the symmetric group $\SN_N$ defines
a game of best choice as follows.  

\begin{definition}\label{d:pflat}
Given a sequence of $i$ distinct integers, we define its {\bf flattening} to be
the unique permutation of $\{1, 2, \ldots, i\}$ having the same relative order
as the elements of the sequence.  Given a permutation $\pi$, define the {\bf
$i$th prefix flattening}, denoted $\pf{i}$, to be the permutation obtained by
flattening the sequence $\pi_1, \pi_2, \ldots, \pi_i$.  In the {\bf game of
best choice} some $\pi \in \I$ is chosen uniformly randomly (with probability
$1/|\I|$) and each prefix flattening $\pf{1}, \pf{2}, \ldots$ is presented
sequentially to the player.  If the player stops at value $N$, they win;
otherwise, they lose.  We refer to $\I$ the set of {\bf interview rank
orderings} for the game.
\end{definition}

Observe that Definition~\ref{d:pflat} easily extends to the case where $\I$ is
a multiset, in which a given permutation may be contained multiple times.  In
this case, the ``uniform'' selection of an interview rank ordering from $\I$
can be used to implement any given discrete probability distribution on
$\SN_N$.  In the following results, we will continue to use traditional subset
notation for simplicity even though we do allow the possibility that $\I$
includes permutations with multiplicity.  The cardinality $|\I|$ of a multiset
$\I$ is the sum of the multiplicities of its elements.

Each strategy for this game can be expressed as a collection of prefixes we
call the {\bf strike set}.  To play the strategy, simply continue interviewing
candidates until one of the prefix flattenings for $\pi$ lies in the strike
set.  Then, select the current candidate and end the game.  It is clear that
every strategy can be represented as a strike set since the player has only the
relative ranking information captured in the prefix flattenings to guide their
stopping decision.

Recall that $\pi_j$ is a {\bf left-to-right maximum} in a permutation $\pi$ if
$\pi_i < \pi_j$ for all $i < j$.  In particular, $\pi_1$ and the entry
containing the value $N$ are always left-to-right maxima in any permutation.
Say that a prefix in $\bigcup_{i=1}^N \SN_i$ is {\bf eligible} if it ends in a
left-to-right maximum.  (Observe that it is never optimal to include a
non-eligible prefix in any strike set.)  We say that an eligible prefix $q$
{\bf contains} an eligible prefix $p$ if the first $i$ entries of $q$ are in
the same relative order as the entries of $p$, where $i$ is the size of $p$.
This defines a partial order on the set of eligible prefixes $\P_N$, which we
refer to as the {\bf prefix tree}.

The prefix tree encodes strategies for the game of best choice.  Specifically,
any optimal strike set must be a {\bf maximal antichain} in $\P_N$ (i.e. a
subset $A$ of prefixes having no pair such that one contains the other, and
such that every prefix in $\P_N$ contains or is contained in some element of
$A$).  The first few prefix trees are shown in Figure~\ref{f:pts} with their
recursive structure highlighted in bold.  Namely, the prefix tree $\P_N$ is
always obtained from the tree $\P_{N-1}$ by appending as leaves the $(N-1)!$
permutations in $\SN_N$ that end with $N$.

\begin{figure}[th]
\[ 
    \scalebox{0.6}{ \begin{tikzpicture}[grow=up]
     \tikzstyle{level 1}=[sibling distance=40mm]
    \node {1}
     child{ node {12} };
\end{tikzpicture}
} \hspace{0.25in}
\scalebox{0.6}{ \begin{tikzpicture}[grow=up]
     \tikzstyle{level 1}=[sibling distance=40mm]
     \tikzstyle{level 2}=[sibling distance=20mm]
    \node {\bf 1}
     child{ node {\bf 12}
       child{ node {123} 
   }
 }
       child{ node {213} };
\end{tikzpicture}
} \hspace{0.25in}
\scalebox{0.6}{ \begin{tikzpicture}[grow=up]
     \tikzstyle{level 1}=[sibling distance=40mm]
     \tikzstyle{level 2}=[sibling distance=20mm]
     \tikzstyle{level 3}=[sibling distance=10mm]
    \node {\bf 1}
     child{ node {\bf 12}
       child{ node {\bf 123} 
       child{ node{1234}}
   }
       child{ node{1324}}
       child{ node{2314}}
 }
       child{ node {\bf 213}
       child{ node{2134}}
        }
      child{ node{3124}}
      child{ node{3214}};
\end{tikzpicture}
}
\]
\caption{Prefix trees $\P_2, \P_3$ and $\P_4$}\label{f:pts}
\end{figure}

We say that a multiset $\I \subseteq \SN_N$ is {\bf strategy-indifferent} if it
defines a game of best choice for which {\em every} strike set, represented by
a maximal antichain, has the same probability of success.  In earlier work
\cite{fowlkes,jones18}, we observed that the subset of $231$-avoiding
permutations in $\SN_N$ is always strategy-indifferent with a win probability
of $\frac{C_{N-1}}{C_N} = \frac{N+1}{4N-2}$, where $C_N$ is the $N$th Catalan
number.  This was the only known example of a strategy-indifferent game prior
to this work.  

\subsection{Game calculus using the prefix tree}
Here, we explain how to compute the win probability for a strategy represented
by a particular strike set.

Observe that each permutation $\pi \in \I$ is winnable for a unique prefix $p$
and define the {\bf strike projection} $\s: \I \ra \P_N$ by $\s(\pi) = p$.
More precisely, $p$ is obtained from $\pi$ by flattening the entries up to and
including the value $N$.  The strike projection allows us to evaluate the
probability of a win using various strategies.  

\begin{lemma}\label{l:wins}
In a game of best choice defined by $\I$, the probability of winning when using the
strategy represented by a maximal antichain $A$ is $\frac{1}{|\I|} \sum_{p \in
A} |\s^{-1}(p)|$.
\end{lemma}
\begin{proof}
Since $|\s^{-1}(p)|$ is the number of interview rank orderings where the player
will select the best candidate when $p$ is included in their strike set, and
each of these preimages are disjoint, we have that the total number of wins for
a strike set $A$ is $\sum_{p \in A} |\s^{-1}(p)|$.  Dividing this sum by $|\I|$
therefore yields the win probability for the strategy represented by $A$.
\end{proof}

\subsection{Characterizing the strategy-indifferent games}
In a typical game theoretic analysis, we would fix $\I$ and try to determine
which strike set $A$ produces the maximal number of wins.  In this work,
however, we are characterizing the $\I$ multisets for which the number of wins
remains constant over {\em all} maximal antichains in $\P_N$.  

To state our main result, we associate terminology from the prefix tree $\P_N$
to multisets of $\SN_N$ via the strike projection $\s$.  For example, we say
that a subset $J \subseteq \SN_N$ {\bf is a maximal saturated chain} if the
strike projection restricts to a bijection from $J$ onto a maximal saturated
chain $\s(J)$ in the prefix tree poset (i.e. a set of prefixes that are totally
ordered by containment to which no prefix could be added without losing the
property of being totally ordered).  We also say that a multiset $\I \subseteq
\SN_N$ {\bf has a maximal chain partition} if it can be partitioned into
distinct subsets that are each maximal saturated chains.  If $\I$ uses repeated
elements, this means that the sum of the multiplicities in the partition for a
given element must equal the corresponding multiplicity for that element in
$\I$.  Recall that in a partial order, $q$ {\bf covers} $p$ if $q$ contains $p$
and no other element is contained between $q$ and $p$.  

\begin{theorem}\label{t:sn_main}
The following are equivalent:
\begin{enumerate}
    \item $\I$ defines a strategy-indifferent game of best choice.
    \item $\I$ has a maximal chain partition.
    \item $|\s^{-1}(p)| = \sum\limits_{\text{$q$ covers $p$}} |\s^{-1}(q)|$ for all non-maximal $p \in \P_N$.
\end{enumerate}
\end{theorem}
\begin{proof}
Suppose $\I$ has a maximal chain partition.  Then $\I$ has the form of a
disjoint union $J_1 \sqcup J_2 \sqcup \cdots \sqcup J_k$ where each $\s(J_i)$
is a maximal saturated chain in $\P_N$.  Therefore, whenever $p \in \s(J_i)$ is non-maximal,
we have $q \in \s(J_i)$ for precisely one $q$ covering $p$ in $\P_N$.  Hence,
\begin{equation}\label{e:pc}
|\s^{-1}(p)| = \sum\limits_{\text{$q$ covers $p$}} |\s^{-1}(q)|  \ \ \text{ for all non-maximal $p \in \P_N$ }
\end{equation}
holds for the restriction of $\s$ to each $J_i$, so it holds for the union
$\I$.  When $\I$ is a multiset, Equation~(\ref{e:pc}) is an equality of sums of
multiplicities and the result follows.

Conversely, suppose that Equation~(\ref{e:pc}) holds.  Among the prefixes $p$
such that $|\s^{-1}(p)| > 0$, choose one that is containment-maximal.  Observe
that this prefix $p$ must also be a maximal element of $\P_N$ or else the $q$
covering $p$ would all have $|\s^{-1}(q)| = 0$ whence $|\s^{-1}(p)|$ would be
$0$ by Equation~(\ref{e:pc}).  By the same equation, we have that the
$|\s^{-1}(r)|$ values are weakly increasing as we move down the poset, from the
maximal element $p$ back towards the root prefix $1$.  Hence, $\I$ must contain
a maximal saturated chain $\mathcal{C}$ starting with the maximal element $p$.
(Here, we are committing a slight abuse of notation in the sense that $p$ is
simultaneously a maximal prefix from $\P_N$ as well as an interview rank
ordering from $\I$.) Observe that if we remove this chain, replacing $\I$ by
$\I \setminus \mathcal{C}$ (interpreted as reducing the multiplicity of each
element from $\mathcal{C}$ by one in $\I$), then Equation~(\ref{e:pc}) will
still hold for $\s$ restricted to $\I \setminus \mathcal{C}$.  So, we may
remove the chain and work by induction on the size of $\I$ to eventually obtain
a maximal chain partition for $\I$.

Thus, we have shown that conditions (2) and (3) are equivalent.  We claim that
(1) and (3) are also equivalent.  If (3) fails for some non-maximal $p$, we
could add some prefixes $B$ to $p$ to create a maximal antichain.  Then, the
union of $B$ with the covers of $p$ would be a distinct strike set with a
different win probability, so $\I$ is not strategy-indifferent in this case.
Conversely, suppose that (3) holds for all $p$.  Given any maximal antichain
$A$, we can successively replace prefixes with the set of their covers without
changing the probability of winning.  Hence, every maximal antichain has the
same win probability as the particular antichain formed from the maximal
elements of $\P_N$, so $\I$ is strategy-indifferent.
\end{proof}

\subsection{Consequences of the Characterization Theorem}
In this section, we derive two important Corollaries of Theorem~\ref{t:sn_main}
that allow us to recast the strategy-indifferent games of best choice in a more intuitive way.

\begin{corollary}\label{c:form}
Consider a strategy-indifferent game of best choice represented by $\I \subseteq \SN_N$.
Let $\J = \{ \pi \in \I | \pi_N = N \}$.  Then, the win probability is given by ${|\J|}/{|\I|}$ which we can also write as
\[ \frac{\sum\limits_{\pi \in \J} 1}{\sum\limits_{\pi \in \J} \#\text{left-to-right maxima}(\pi)}. \]
\end{corollary}
\begin{proof}
By Theorem~\ref{t:sn_main}, we have seen that every strategy-indifferent game
of best choice arises from a collection of permutations that can be decomposed
into maximal saturated chains.  Each chain can be represented by its maximal
element, which necessarily ends with value $N$.  Then, the maximal element in a
chain, say $\pi \in \J$, contains left-to-right maxima precisely at the ending
positions of each eligible prefix along the chain.  Consequently, $|\I|$ can be
computed by summing the number of left-to-right maxima in each $\pi \in \J$
(with multiplicity if necessary).  Since the collection $\s(\J)$ of maximal
elements from the maximal saturated chains forms an antichain in the prefix
tree, it is also a valid strike set with $|\J|$ wins.  Dividing these, we
obtain the win probability by Lemma~\ref{l:wins}.
\end{proof}

Since we can recover each maximal saturated chain in $\P_N$ from its maximal
element, we see that all of the numerical information from the strike
projection $\s : \I \rightarrow \P_N$ is actually already contained in the
sub-multiset $\J \subset \I$.  We formalize this as follows.  Say that two
games of best choice $\I, \I' \subseteq \SN_N$ are {\bf strategically
equivalent} if for all maximal antichains $A \subseteq \P_N$ we have that $A$
produces the same win probability for $\I$ as it does for $\I'$. 
We are now in position to define a central concept for this work.  

\begin{corollary}\label{c:const}
Every strategy-indifferent game of best choice is strategically equivalent to a
game $\I'$ arising from a multiset $\J'$ of $\SN_{N-1}$ by the following construction:
\begin{itemize}
\item[(1)] Let $\J$ be the multiset of permutations in $\SN_N$ obtained by appending $N$ to the end of each permutation in $\J'$, and
\item[(2)] for each $\pi \in \J$ and each left-to-right maximum position $i$ of $\pi$, add a new permutation $\widetilde{\pi}$ to $\I'$ such that $\widetilde{\pi}|_{[i]} = \pi|_{[i]}$ and $\widetilde{\pi}_i = N$.
\end{itemize}

Moreover, every multiset $\J' \subseteq \SN_{N-1}$ defines in this way a
strategy-indifferent game of best choice. 
\end{corollary}
\begin{definition}
We call the multiset $\J' \subseteq \SN_{N-1}$, as in Corollary~\ref{c:const}, the {\bf
bine of competitors} for the associated strategy-indifferent game of best
choice.  
\end{definition}
At the level of the game, the permutations in this multiset represent
the possible relative interview rank orderings for all of the candidates that
are ``competing'' with the best candidate.  Structurally, these multisets give
us a canonical way to represent strategy-indifferent games, up to strategic
equivalence, and we can ``grow'' a full set of interview rank orderings $\I
\subseteq \SN_N$ from each such multiset $\J' \subseteq \SN_{N-1}$.

\begin{proof}
Given a strategy-indifferent game $\I$, we let 
\[ \J' = \{ \pi|_{[N-1]} \in \SN_{N-1} \text{ where } \pi \in \I \text{ and } \pi_N = N \}. \]  
We know $\I$ has a maximal chain partition by Theorem~\ref{t:sn_main} so we may
work one chain at a time to verify that, by construction, the game $\I'$ arising
from the bine of competitors $\J'$ will have the same $|\s^{-1}(p)|$ values as for the
original game with $\I$.  However, the precise choice of permutations for $\I'$
can vary from $\I$.  For example, the construction of $\widetilde{\pi}$ does not specify
which entries should occur in positions after the value $N$ nor which entries
should be used to effect the desired relative order in the positions before $N$
(because these choices are irrelevant for the game).  Since the $|\s^{-1}(p)|$
values represent the number of wins for a given prefix $p$, the two games are
strategically equivalent.

Moreover, given any multiset $\J'$ of $\SN_{N-1}$, the construction in the
statement will produce a collection of $|\J'|$ maximal saturated chains,
thereby creating a strategy-indifferent game by Theorem~\ref{t:sn_main}.
\end{proof}

Thus, the strategy-indifferent games of best choice are in one-to-one
correspondence with the multisets $\J'$ of $\SN_{N-1}$.  To help interpret
this, imagine that we described the classical game of best choice as follows:
given a uniformly selected permutation from $\SN_{N-1}$ representing the
relative interview rank orderings of all but the best candidate, a position is
selected uniformly from $\{1, \ldots, N\}$ in which to place the best
candidate.  This process clearly results in a uniform distribution on $\SN_N$.
In our case, for a given permutation from the bine of competitors (this is no
longer a uniform selection but weighted according to the number of
left-to-right maxima), the positions for the best candidate in a
strategy-indifferent game are selected uniformly from the set of left-to-right
maxima.

We now turn to the question of interpreting the permutations that arise in
$\I$.  One canonical way to choose permutations for the maximal saturated
chains that make up the $\I$ multiset, given a bine of competitors $\J'$, is to
iteratively apply the following map to each of the elements of $\J$.  This is
the unique construction that preserves the relative ordering of all but the
best candidate, so is especially suited to sets defined by pattern-avoidance
criteria, and we used a restricted version of it in our proof of the
$231$-avoiding result from \cite{jones18}.

\begin{definition}\label{d:map}
Given a permutation $\pi \in \SN_N$ with at least two left-to-right maxima, let
$\varphi$ be the result of placing $N$ in the position of the next to last
left-to-right maximum in $\pi$ and keeping all other values in the same
relative order.  
\end{definition}

Observe that this process is reversible:  if we plot the left-to-right maxima
of the permutation that results from removing $N$ from $\pi$ then we can place
$N$ in the position of the next left-to-right maximum to the right, keeping all
other values in the same relative order.  Consequently, if $\J'$ is
multiplicity-free then so is the $\I$ that is built using $\varphi$.
Regardless, the permutations that result from this construction always have the
property that {\em the entry immediately right of value $N$ is always larger in
value than all entries lying to the left of value $N$}. 

\begin{remark}\label{r:231}
It is interesting to note that avoiding the pattern $231$ automatically
enforces this condition without mentioning the value $N$ explicitly.
\end{remark}

\begin{remark}\label{r:map}
Another natural construction would be to place $N$ in the position of the next
to last left-to-right maximum by moving this entry to the last position, and
fixing all of the other entries, but we do not pursue this here.  
\end{remark}

By Corollary~\ref{c:form}, we find that the win probability for any
strategy-indifferent game can be computed from the number of left-to-right
maxima in the permutations from the bine of competitors $\J'$.  More precisely,
given a multiset $\J'$, define 
\[ E_\mathsf{LRM}(\J') = \frac{1}{|\J'|} \sum_{\pi \in \J'} \#\text{left-to-right maxima}(\pi), \]
the expected number of left-to-right maxima for a uniformly random element of
$\J'$ (so selected in proportion with its multiplicity if $\J'$ is a multiset).
Then we have the following attractive formulas.
\begin{lemma}\label{l:nums}
For the strategy-indifferent game $\I$ that uses $\J' \subseteq \SN_{N-1}$ as
its bine of competitors, we have that the win probability is
$1/(E_\mathsf{LRM}(\J') + 1) = 1/E_\mathsf{LRM}(\J)$ and $|\I| = |\J'| (E_\mathsf{LRM}(\J') + 1)$.
\end{lemma}
\begin{proof}
Compare $E_\mathsf{LRM}(\J')$ with Corollary~\ref{c:form} and resolve the
difference in rank between $\J$ and $\J'$.
\end{proof}

\begin{example}
As a synthetic example, suppose that $\J' = \{ {\dot 3}1{\dot 4}2, {\dot 1}{\dot 4}23, {\dot 2}{\dot 3}1{\dot 4} \}$, where we have marked the left-to-right maxima and $E_\mathsf{LRM}(\J') = 7/3$.
Then $\J = \{ 31425, 14235, 23145 \}$ and we may grow $\I$ using $\varphi$ as $\{ 31425, 31542, 53142 \} \cup \{ 14235, 15423, 51423 \} \cup \{ 23145, 23154, 25314,$ $52314 \} $.  Other alternatives for $\I$, following the construction in Corollary~\ref{c:const}, all have the form 
\[ \{ 31425,\  [21] 5 \q\ \q,\  5 \q\ \q\ \q\ \q \} \cup \{ 14235,\  [1] 5 \q\ \q\ \q,\  5 \q\ \q\ \q\ \q \} \cup \{ 23145,\  [231] 5 \q,\  [1] 5 \q\ \q\ \q,\  5 \q\ \q\ \q\ \q \} \]
where the values in brackets are specified only up to relative ordering.  
In any case, the prefix tree (with $|\s^{-1}(p)|$ values shown in parentheses) looks like
\[ 
\scalebox{0.75}{ \begin{tikzpicture}[grow=right]
     \tikzstyle{level 1}=[sibling distance=13mm]
     \tikzstyle{level 2}=[sibling distance=8mm]
     \tikzstyle{level 3}=[sibling distance=10mm]
     \tikzstyle{level 4}=[sibling distance=5mm]
    \node {$1 \atop (3)$}
     child{ node {$12 \atop (2)$}
       child{ node {$2314 \atop (1)$} 
       child{ node{{\bf $\mathbf{23145} \atop (1)$ }}}
   }
       child{ node{\bf $\mathbf{14235} \atop (1)$}}
 }
       child{ node {$213 \atop (1)$}
       child{ node{ {\bf $\mathbf{31425} \atop (1)$}}}
        };
\end{tikzpicture}
}
\]
and any strike set for this game will produce the same win probability, namely 
$\frac{1+1+1}{3+3+4} = \frac{3}{10} = \frac{1}{\frac{7}{3} + 1}$.
\end{example}

\section{Examples}\label{s:lrms}

To this point, we have focused on games in a fixed rank $N$.  The classical
game of best choice allows for any number of candidates, however, and the
analysis proceeds by considering the limiting strategies and win probabilities
as $N \ra \infty$.  So, we consider collections of subsets $\J_i' \subseteq
\SN_i$ ($i \geq 1$), that we use as bines of competitors to create
strategy-indifferent games of best choice in each rank $i+1$.  We are
particularly interested in examples for which $\lim_{N \ra \infty}
E_\mathsf{LRM}(\J_N')$ is a finite value, because when the expected value grows
without bound the win probability goes to $0$.  

A summary of our results are presented in Figure~\ref{f:results}.  These
prominently feature the Catalan numbers, which we are denoting $C_N$.  Here,
the third column records the number of permutations in $\J' \subseteq
\SN_{N-1}$ having exactly $k$ left-to-right maxima, yielding generalized
``Stirling triangles'' for sets of permutations defined by avoiding a single
pattern of size $3$ (see \cite{oeis}).  The entries in the second column are classical; see
\cite{bona}, for example.  The entries in the last column are obtained from the
entries in the penultimate column via Lemma~\ref{l:nums} and a limit.  One may
also use Lemma~\ref{l:nums} and our expressions for $E_\mathsf{LRM}(\J')$ to
compute $|\I|$ for each game.  In the remainder of this section, we justify the
other enumerations and equivalences.

\begin{figure}[h]
\renewcommand{\arraystretch}{2.2}
    \scalebox{0.75}{
\begin{tabular}{|l|l|l|l|l|}
    \hline 
    $\J'$ & $|\J'|$ & $\mathsf{LRM}$ multiplicities & $E_\mathsf{LRM}(\J')$ & asymptotic win probability \\
    \hline 
    \hline 
    $\SN_{N-1}$ & $(N-1)!$ & Stirling numbers (first kind) $\bchoose{N-1}{k}$ & $\sum_{i=1}^{N-1} \frac{1}{i} \approx \ln(N)$ & 0 \\
    \hline
    $123$-avoiding &  $C_{N-1}$ & two-columns: $(C_{N-2}, C_{N-1}-C_{N-2})$ & $2-\frac{C_{N-2}}{C_{N-1}} = \frac{7N-5}{4N-2} $ & 4/11 $\approx 36.4\%$ \\
    \hline
    \parbox{1.9in}{$231$-avoiding $\cong$ $132$-avoiding $\cong$ $213$-avoiding} & $C_{N-1}$ & Catalan ballot triangle {\bf (A033184)} & $\frac{C_N}{C_{N-1}} - 1 = \frac{3N}{N+2}$ & 1/4 $= 25\%$ \\
    \hline
    $321$-avoiding $\cong$ $312$-avoiding & $C_{N-1}$ & Catalan--Narayana triangle {\bf (A001263)} & $(2N-3) \frac{C_{N-2}}{C_{N-1}} = \frac{N}{2}$ & 0 \\
    \hline 
\end{tabular} }
\caption{Enumerative results for some strategy-indifferent games}\label{f:results}
\end{figure}

\subsection{$\J'$ is the full symmetric group $\SN_{N-1}$}

Here, the competitors of the best candidate may arrive in any order, and the
implied probability distribution will weight each permutation in proportion
with its number of left-to-right maxima.  The number of permutations in rank
$N-1$ with exactly $k$ left-to-right maxima is the Stirling number of the first
kind $\bchoose{N-1}{k}$; see e.g. \cite{wilf,renyi}.  Therefore, the expected
number of left-to-right maxima in rank $N-1$ is
\[ E_\mathsf{LRM}(\J') = \frac{1}{(N-1)!} \sum_{k=1}^{N-1} \bchoose{N-1}{k} k = \sum_{i=1}^{N-1} \frac{1}{i}, \]
the $(N-1)$st Harmonic number.  Although the second equality is well-known, it may also be derived from the following
result for which we include a bijective proof to give a self-contained treatment.

\begin{lemma}\label{l:comb}
We have 
\[ \sum_{i=1}^{N-1} (1+i) \left[{ {N-1} \atop i }\right] = (N-1)! \left(1 + \sum_{i=1}^{N-1} \frac{1}{i} \right).\]
\end{lemma}
\begin{proof}
Using Definition~\ref{d:map} to implement the construction in
Corollary~\ref{c:const}, the expression on the left counts the number of
permutations in $\SN_N$ having the property that the entry immediately right of
value $N$ is always larger in value than all entries lying to the left of value
$N$, which we refer to here as {\bf $(2-N1)$-avoiding}.  The expression on the
right counts permutations in $\SN_N$ with at most two left-to-right maxima by
summing over all of the possible positions for the value $N$.  More precisely,
suppose $N$ lies in some position $k > 1$.  Then, there are $(k-1)!$ total
permutations of the values in positions to the left of $N$, of which $(k-2)!$
will produce exactly two left-to-right maxima by virtue of having their maximal
value in the first position.  Thus, $\frac{(k-2)!}{(k-1)!} = \frac{1}{k-1}$ of
the $(N-1)!$ possible permutations with $N$ in position $k$ have precisely two
left-to-right maxima.  There is exactly one left-to-right maximum if and only
if $N$ lies in the first position and there are $(N-1)!$ permutations that meet
this condition.  The given formula follows.

To prove the equation, we provide a bijection between these sets.  We do not
need to worry about permutations that meet both conditions, so restrict to
permutations with exactly two left-to-right maxima, $y$ and $N$, with $N$
participating in a $(2-N1)$ instance.  
These have the form
\[ \dot{y} \underbrace{\cdots}_{<y} \dot{N} \underbrace{x_1 x_2 \cdots x_k}_{<y} \underbrace{\mathring{z}}_{>y} \cdots \]
for some $k \geq 1$.  (If $z$ does not exist, we just let $x_1 x_2 \cdots x_k$ be all the entries to the right of $N$.)

If we send this to 
\[ \underbrace{\dot{x_1} x_2 \cdots x_k}_{<y} \dot{y} \underbrace{\cdots}_{<y} \dot{N} \underbrace{\mathring{z}}_{>y} \cdots \]
we will have a $(2-N1)$-avoiding permutation with more than two left-to-right
maxima since $x_1$ becomes a new left-to-right maxima and some of the other
$x_i$'s may become left-to-right maxima (not shown).

Moreover, this process is reversible:  take the block up to the penultimate
left-to-right maximum and put it just after $N$.  Thus, we have a bijection and
the equality of counting formulas follows.
\end{proof}

Note that, by Lemma~\ref{l:nums}, we have that $|\I|$ is equal to either of the
expressions in Lemma~\ref{l:comb}.  This sequence begins $1, 2, 5, 17, 74,
\ldots$ and appears as {\bf (A000774)} in the OEIS \cite{oeis}.  Unfortunately,
since the harmonic series diverges, the reciprocal win probability is tending
to $0$.  This game does have the largest possible size $|\I|$ for any
strategy-indifferent game that does not repeat interview rank orderings and one
may use Definition~\ref{d:map} or Remark~\ref{r:map} to explicitly build such
an $\I$ set.

\subsection{$\J'$ is the subset of 123-avoiding permutations}

Avoiding the permutation $123$ requires the competitors to be decreasing in
quality as time passes, as could occur in a market with deteriorating extrinsic
conditions.  More precisely, these are the permutations with at most two
left-to-right maxima and, in order to be strategy-indifferent, the best
candidate must occur at one of these positions.

Although the permutations from $\J'$ have rank $N-1$, we apply counting
formulas in rank $N$ to relate our work with existing results from the
literature.  The number of $123$-avoiding permutations with a single
left-to-right maximum in rank $N$ is the same as the number of $123$-avoiding
permutations with $N$ in the first position, which is the same as the number of
$123$-avoiding permutations of rank $N-1$ which is the Catalan number
$C_{N-1}$.  Since there cannot be more than two left-to-right maxima in any
$123$-avoiding permutation, the number of $123$-avoiding permutations with
exactly two left-to-right maxima is therefore $C_N - C_{N-1}$.

Thus, the expected number of left-to-right maxima in rank $N$ is
\[ \frac{1}{C_{N}}( C_{N-1} + 2 (C_{N} - C_{N-1}) ) = 2 - \frac{C_{N-1}}{C_{N}} \]
and the win probability $1/(E_\mathsf{LRM}(\J')+1)$ is tending to $1/(2-1/4+1) =
1/(11/4) = 4/11 \approx 36\%$.  This is the highest win probability among the
strategy-indifferent games we have examined so far, and is nearly as good as
the classical $1/e$ win probability.

\subsection{$\J'$ is the subset of 231-avoiding permutations}\label{s:22}

As mentioned in the introduction, this example was treated in prior work.
Avoiding the permutation $231$ requires a sequence of interviews to be
``status-seeking'' in the sense that each time an interview is ranked more
highly than some previous interview, a floor is set for all future candidate
rankings.  Also, this pattern is unique in that avoiding it in the bine of
competitors actually carries over to the same avoidance criterion in the full
subset $\I$ of interview rank orderings.  In \cite[Section 3]{jones18}, the win
probability for this game was found to be $\frac{C_{N-1}}{C_N}$.  Applying
Lemma~\ref{l:nums} then gives a novel way to compute $E_\mathsf{LRM}$.
The bijection between the ``prefix trees'' in \cite[Section 3]{jones18}
preserves left-to-right maxima so this analysis applies to the bine of
competitors avoiding $132$ (although the resulting subset $\I$ is not
characterized by avoiding the same pattern).

In order to show that the $231$-avoiding bine of competitors is isomorphic to
the $213$-avoiding bine of competitors, we need a bijection that preserves
left-to-right maxima.  The idea is to consider the recursive structure that
defines each permutation class.  Given a $231$-avoiding permutation, we
partition the entries based on the value $k$ of the entry in the first
position, where every entry with values larger than $k$, which we call the {\bf
upper block}, lies to the right of every entry with values smaller than $k$,
which we call the {\bf lower block}.  These blocks themselves are (shifted)
$231$-avoiding permutations and each block has an initial element which can be
used to partition the block into upper and lower sub-blocks.  Repeating this
procedure produces a recursive decomposition of the $231$-avoiding permutation
that terminates when all of the (sub-)blocks are empty.  In a completely
similar way, we can decompose a $213$-avoiding permutation based on the value
of the entry in the first position but where every entry in the upper block
must now lie to the left of every entry in the lower block.  

Consequently, we define a function that fixes the initial element, interchanges
the positions of the upper and lower blocks, and recursively applies to each of
the blocks individually.  See Figure~\ref{f:mgf} for an illustration.  It is
clear that this operation is reversible and that it translates the pattern
criteria from $231$-avoiding in the domain to $213$-avoiding in the image.  It
is also not difficult to see that the values of the left-to-right maxima, which
are the initial elements of the upper blocks, remain unchanged.  Hence, the
strategy-indifferent game arising from the $231$-avoiding bine of competitors
is isomorphic to the game arising from the $213$-avoiding bine of competitors.

\begin{figure}[ht]
\begin{center}
\begin{tikzpicture}[scale=0.20]
    \draw[dashed] (0.95,12) -- (24,12);
    \draw (13,23) -- (23,23) -- (23,13) -- (13,13) -- (13,23);
    \draw (6,6) node {$B$};
    \draw (18,18) node {$A$};
    \draw (1,11) -- (11,11) -- (11,1) -- (1,1) -- (1,11);
    \draw [dotted] (0,0) -- (0,24) -- (24,24) -- (24,0) -- (0,0);
    \node at (0,12) [circle,inner sep=1pt,minimum size=1pt,draw,fill=lightgray!90] {$k$};
\end{tikzpicture}
\hspace{0.3in}
\begin{tikzpicture}[scale=0.20]
    \draw (0,11) node {$\longmapsto$};
    \draw (0,0) node {$ $};
\end{tikzpicture}
\hspace{0.3in}
\begin{tikzpicture}[scale=0.20]
    \draw[dashed] (0.95,12) -- (24,12);
    \draw (13,11) -- (23,11) -- (23,1) -- (13,1) -- (13,11);
    \draw (6,18) node {$f(A)$};
    \draw (18,6) node {$f(B)$};
    \draw (1,13) -- (1,23) -- (11,23) -- (11,13) -- (1,13);
    \draw [dotted] (0,0) -- (0,24) -- (24,24) -- (24,0) -- (0,0);
    \node at (0,12) [circle,inner sep=1pt,minimum size=1pt,draw,fill=lightgray!90] {$k$};
\end{tikzpicture}
\end{center}
\caption{Recursive bijection between $231$-avoiding and $213$-avoiding permutations that preserves (values of) left-to-right maxima}\label{f:mgf}
\end{figure}
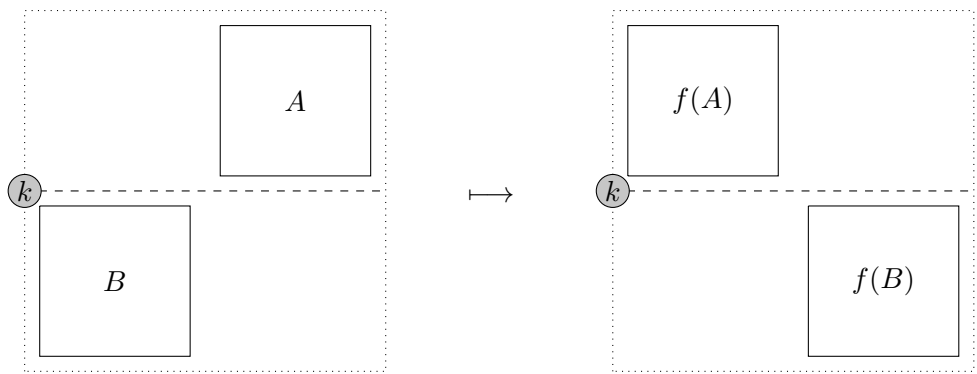

\subsection{$\J'$ is the subset of 321-avoiding permutations}

Avoiding the permutation $321$ requires the competitors to be increasing in
quality as time passes, as could occur in a market with improving extrinsic
conditions.  More precisely, these are the permutations with at most two
left-to-right minima.  We also characterized these as ``disappointment-free''
in \cite{jones18}, in the sense that whenever a candidate ranks lower than an
earlier candidate, that earlier candidate's rank becomes a floor for all future
rankings.

Reviewing some material from \cite[Section 4]{fowlkes}, a $321$-avoiding
permutation is completely determined by the values and positions of its
left-to-right maxima, and these are encoded by Dyck paths.  Thus, the number of
left-to-right maxima in a particular permutation is equal to the number of
corners or ``peaks'' in the corresponding Dyck path.  The multiplicities for
these are well-known (e.g. \cite[Ex. 6.36]{ec2}) to be counted by the Catalan--Narayana triangle,
$T(N,k) = \frac{1}{k}{ {N-1} \choose {k-1} }{ N \choose {k-1} }$, having the
property that $\sum_{k=1}^N T(N,k)$ is the $N$th Catalan number $C_N$.

To compute the expected number of left-to-right maxima in rank $N$ as
\[ \frac{1}{C_N} \sum_{k=1}^N k T(N,k) = \frac{1}{C_N} \sum_{k=1}^N { {N-1} \choose {k-1} }{ N \choose {k-1} } = \frac{(2N-1)C_{N-1}}{C_N} = \frac{(2N-1)(N+1)}{2(2N-1)} = \frac{N+1}{2}, \]
it suffices to show that 
\begin{equation}
\label{e:321}
\sum_{k=1}^N { {N-1} \choose {k-1} }{ N \choose {k-1} } = { {2N-1} \choose N } 
\end{equation}
because
\[ { {2N-1} \choose N } = \frac{2N-1}{N} {{2N-2} \choose {N-1}} = (2N-1)C_{N-1}. \]

To see that Equation~(\ref{e:321}) holds, observe that both of these
expressions count the number of ways to put $N$ indistinguishable stars into
$N$ distinguishable bins.  The expression on the right results from the classic
``stars and bars'' argument.  Namely, it counts the number of ways to select
$N$ positions for ``stars'' from the $2N-1$ positions total, where we let the
rest of the positions be $N-1$ ``bar'' separators between the $N$ bins.  The
expression on the left considers the first $N$ positions separately from the
last $N-1$ positions.  Initially, when $k-1 = 0$, we have the configuration of
all stars followed by all bars (so all of the stars belong to the first bin).
For positive values of $k-1$, we count the configurations where $k-1$ stars and
$k-1$ bars have been interchanged from one set of positions to the other.  In
this way, we are also considering every possible string of $N$ stars and $N-1$
bars.

Finally turning to the $312$-avoiding permutations, similar remarks to
\cite[Section 4]{fowlkes} apply.  Namely, once we know the values and positions
of the left-to-right maxima in a $312$ avoiding permutation, all of the
complementary values must appear in decreasing order.  So these are again
encoded by Dyck paths, and our analysis above applies verbatim.

\bigskip
\section*{Acknowledgements}

This project was begun at the 2019 Research Experiences for Undergraduate
Faculty (REUF) program at the Institute for Computational and Experimental
Research in Mathematics (ICERM), supported by the National Science Foundation
through DMS grant 1620080.  We are grateful to everyone at REUF / ICERM for
catalyzing and supporting this project.  We also thank Alex Burstein for advice
regarding the bijection in Section~\ref{s:22}.


\bibliographystyle{amsalpha}
\bibliography{bestchoice}

\end{document}